\documentclass[a4paper,twoside,12pt]{article}
\usepackage[english]{babel}

\usepackage{amsmath,amsfonts,amssymb,amsthm}
\newtheorem{teo}{Theorem}
\newtheorem{lem}{Lemma}

\newtheorem{remark}{Remark}

 \title{{\bf     Abstract    evolution equations with an operator function in the second term   }}

\author{Maksim \,V.~Kukushkin   \\ \\
 \small  \textit{Moscow State University of Civil Engineering, 129337,  Moscow, Russia}\\
 \small\textit{Kabardino-Balkarian Scientific Center, RAS, 360051,  Nalchik, Russia}\\
\textit{\small\textit{kukushkinmv@rambler.ru}} }

\date{}
\begin{document}

\maketitle

\begin{abstract}
In this paper, having introduced a convergence of a series on the root vectors  in the  Abel-Lidskii sense, we present a valuable application to the evolution equations.
The main issue of the paper is an approach  allowing us to principally broaden conditions imposed upon the right-hand side of the evolution equation in the abstract Hilbert space.
 In this way, we come to the definition of the function of an unbounded non-selfadjoint operator. Meanwhile, considering the main issue we involve an additional concept that  is a generalization of the spectral theorem for a non-selfadjoint operator.

\end{abstract}
\begin{small}\textbf{Keywords:}
 Spectral theorem;  Abel-Lidskii basis property;   Schatten-von Neumann  class; operator function; evolution equation.   
\end{small}

\section{Introduction}

  The application  of results connected with the basis property in the Abell-Lidskii sense  \cite{firstab_lit:1kukushkin2021}  covers  many   problems \cite{firstab_lit:2kukushkin2022} in the framework of the theory of evolution  equations. The central idea of this paper is devoted to  an approach  allowing us to principally broaden conditions imposed upon the right-hand side of the evolution equation in the abstract Hilbert space. In this way we can obtain abstract results covering many applied problems to say nothing on the far-reaching   generalizations.
     We plan to implement the idea  having involved a notion of    an operator function. This is why one of the paper challenges is to find a harmonious way of reformulating the main principles of the spectral theorem having  taken  into account the peculiarities of the convergence in the Abel-Lidskii sense. However, our final goal is  an existence and uniqueness theorem for an abstract  evolution equation with an operator function at the right-hand side, where  the derivative at the left-hand side is supposed to be of the integer order. The  peculiar contribution   is the obtained formula for the solution of the evolution equation.   We should remind that involving a notion of the operator function, we broaden a great deal  a class corresponding to the right-hand side.   This gives us an opportunity to claim  that the main issue of the paper is closely connected  with  the spectral theorem for  the non-selfadjoint unbounded  operator. Here, we should make a brief digression and consider a theoretical background that allows us to obtain such exotic results.

    We should recall that the concept of the spectral theorem for a selfadgoint operator is based on the notion of a spectral family or the decomposition of the identical operator. Constructing a spectral family, we can define a selfadjoint  operator using a concept of the  Riemann integral, it is the very statement of the spectral theorem for a selfadjoint operator.   Using the same scheme, we come to a notion of the operator function of a selfadjoint operator.
         The idea can be clearly demonstrated if we consider  the  well-known representation  of  the compact selfadjoint operator as  a series on its  eigenvectors. The case corresponding to a non-selfadjoint operator is not so clear but we can adopt some notions and techniques to obtain similar  results. Firstly, we should note that    the question regarding decompositions of the operator on the series of eigenvectors  (root vectors)   is rather complicated and  deserves to be considered itself. For this purpose, we need to involve some generalized notions of the series convergence, we are compelled to understand it in one or another sense, we mean   Bari, Riesz, Abel (Abel-Lidskii) senses of the series convergence  \cite{firstab_lit:2Agranovich1994},\cite{firstab_lit:1Gohberg1965}.  The main disadvantage of the paper \cite{firstab_lit:2Agranovich1994} is a  sufficiently strong  condition imposed upon the numerical range of values comparatively with the sectorial condition (see definition of the sectorial operator), there considered a domain of the parabolic type containing the spectrum of the operator.  A reasonable question that appears is about minimal conditions that guaranty the desired result.    At the same time  the convergence in the Abel-Lidskii sense  was established in the paper \cite{firstab_lit:1Lidskii} for an operator class wider than the class of sectorial operators. A major contribution of the paper \cite{firstab_lit:1kukushkin2021} to the theory is a  sufficient condition  for the Abel-Lidskii basis property of the root functions system for a sectorial non-selfadjoint operator of the special type. Considering such an operator class we strengthen a little the condition regarding the semi-angle of the sector, but weaken a great deal conditions regarding the involved parameters. Moreover, the central aim generates  some prerequisites to consider technical peculiarities such as a newly constructed sequence of contours of the power type on the contrary to the Lidskii results \cite{firstab_lit:1Lidskii}, where a sequence of the contours of the  exponential type was considered. In any case, we may say that  a  clarification of the    results  \cite{firstab_lit:1Lidskii} devoted to the decomposition on the root vectors system of the non-selfadjoint operator has been obtained. In the paper \cite{firstab_lit:1kukushkin2021}, we used   a technique of the entire function theory and introduced  a so-called  Schatten-von Neumann class of the convergence  exponent.    Using a  sequence of contours of the power type,  we invented a peculiar  method how to calculate  a contour integral, involved in the problem in its general statement,  for strictly accretive operators satisfying special conditions formulated in terms of the norm.  Here, we  used  an equivalence  between operators with the discrete spectra and operators with the compact resolvent, for the results  devoted to  them  can be easily reformulated from one to another realm.

       In order to avoid the lack of information, consider a  class of non-selfadjoint operators for which the above concept can be successfully applied.  We ought to note
          that a peculiar scientific interest appears in the case when a senior term of the operator  is not selfadjoint
           \cite{firstab_lit:Shkalikov A.},\cite{firstab_lit(arXiv non-self)kukushkin2018},  for in the contrary case there is a plenty of results devoted to the topic wherein  the following papers are well-known
\cite{firstab_lit:1Katsnelson},\cite{firstab_lit:1Krein},\cite{firstab_lit:Markus Matsaev},\cite{firstab_lit:2Markus},\cite{firstab_lit:Shkalikov A.}. The fact is that most of them deal with a decomposition of the  operator  on a sum,  where the senior term
     must be either a selfadjoint ore normal operator. In other cases, the  methods of the papers
     \cite{kukushkin2019}, \cite{firstab_lit(arXiv non-self)kukushkin2018} become relevant  and allow us  to study spectral properties of  operators  whether we have the mentioned above  representation or not. Here, we should remark that
the methods  \cite{firstab_lit(arXiv non-self)kukushkin2018}   can be  used    in  the  natural way,  if we deal with abstract constructions formulated in terms of the semigroup theory   \cite{kukushkin2021a}.  The  central challenge  of the latter  paper  is how  to create a model   representing     a  composition of  fractional differential operators   in terms of the semigroup theory.   We should note that motivation arises  in connection with the fact that
  a second order differential operator can be represented  as a some kind of  a  transform of   the infinitesimal generator of a shift semigroup.
 Having been inspired by  the novelty of the  idea,  we generalize a   differential operator with a fractional integro-differential composition  in the final terms   to some transform of the corresponding  infinitesimal generator of the shift semigroup.
Having applied  the   methods
\cite{firstab_lit(arXiv non-self)kukushkin2018}, we   managed  to  study spectral properties of the  infinitesimal generator  transform and obtain   an outstanding result --
   asymptotic equivalence between   the
real component of the resolvent and the resolvent of the   real component of the operator. The relevance is based on the fact that
   the  asymptotic formula  for   the operator  real component  in most  cases  can be  established  due to well-known results  \cite{firstab_lit:Rosenblum}. Thus, the existence and  uniqueness theorems formulated in terms of the operator order \cite{firstab_lit:2kukushkin2022}, subsequently generalized  due to involving the notion of the operator function,  can cover a significantly large operator class.

Apparently, the application part of the paper appeals to the theory of differential equations.  In particular,     the existence and uniqueness theorems  for evolution  equations   with the right-hand side --  an operator function of a  differential operator  with a fractional derivative in  final terms are covered by the invented abstract method. In this regard such operators  as the Riemann-Liouville  fractional differential operator, the Kipriyanov operator,  the Riesz potential, the difference operator, the artificially constructed normal operator  are involved \cite{firstab_lit:2kukushkin2022}.

\section{Preliminaries}

Let    $ C,C_{i} ,\;i\in \mathbb{N}_{0}$ be   real constants. We   assume   that  a  value of $C$ is positive and   can be different within a formula   but   values of $C_{i} $ are  certain. Denote by $ \mathrm{int} \,M,\;\mathrm{Fr}\,M$ the interior and the set of boundary points of the set $M$ respectively.   Everywhere further, if the contrary is not stated, we consider   linear    densely defined operators acting on a separable complex  Hilbert space $\mathfrak{H}$. Denote by $ \mathcal{B} (\mathfrak{H})$    the set of linear bounded operators   on    $\mathfrak{H}.$      Denote by    $    \mathrm{D}   (L),\,   \mathrm{R}   (L),\,\mathrm{N}(L)$      the  {\it domain of definition}, the {\it range},  and the {\it kernel} or {\it null space}  of an  operator $L$ respectively.
  Let    $\mathrm{P}(L)$ be  the resolvent set of an operator $L.$
       Denote by $\lambda_{i}(L),\,i\in \mathbb{N} $ the eigenvalues of an operator $L.$
 Suppose $L$ is  a compact operator and  $N:=(L^{\ast}L)^{1/2},\,r(N):={\rm dim}\,  \mathrm{R}  (N);$ then   the eigenvalues of the operator $N$ are called   the {\it singular  numbers} ({\it s-numbers}) of the operator $L$ and are denoted by $s_{i}(L),\,i=1,\,2,...\,,r(N).$ If $r(N)<\infty,$ then we put by definition     $s_{i}=0,\,i=r(N)+1,2,...\,.$
 According  to the terminology of the monograph   \cite{firstab_lit:1Gohberg1965}  the  dimension  of the  root vectors subspace  corresponding  to a certain  eigenvalue $\lambda_{k}$  is called  the {\it algebraic multiplicity} of the eigenvalue $\lambda_{k}.$
  Denote by $n(r)$ a function equals to the number of the elements of the sequence $\{a_{n}\}_{1}^{\infty},\,|a_{n}|\uparrow\infty$ within the circle $|z|<r.$ Let $A$ be a compact operator, denote by $n_{A}(r)$   {\it counting function}   a function $n(r)$ corresponding to the sequence  $\{s^{-1}_{i}(A)\}_{1}^{\infty}.$
  Let  $\mathfrak{S}_{p}(\mathfrak{H}),\, 0< p<\infty $ be       a Schatten-von Neumann    class and      $\mathfrak{S}_{\infty}(\mathfrak{H})$ be the set of compact operators.
     In accordance with  the terminology of the monograph  \cite{firstab_lit:kato1980} the set $\Theta(L):=\{z\in \mathbb{C}: z=(Lf,f)_{\mathfrak{H}},\,f\in  \mathrm{D} (L),\,\|f\|_{\mathfrak{H}}=1\}$ is called the  {\it numerical range}  of an   operator $L.$
  An  operator $L$ is called    {\it sectorial}    if its  numerical range   belongs to a  closed
sector     $\mathfrak{ L}_{\iota}(\theta):=\{\zeta:\,|\arg(\zeta-\iota)|\leq\theta<\pi/2\} ,$ where      $\iota$ is the vertex   and  $ \theta$ is the semi-angle of the sector   $\mathfrak{ L}_{\iota}(\theta).$ If we want to stress the  correspondence  between $\iota$ and $\theta,$  then   we will write $\theta_{\iota}.$
 By the   {\it convergence exponent} $\rho$ of the sequence
$
\{a_{n}\}_{1}^{\infty}\subset \mathbb{C},\,a_{n}\neq 0,\,a_{n}\rightarrow \infty
$
 we mean the greatest lower bound for such numbers $\lambda$ that the  following   series   converges
$$
 \sum\limits_{n=1}^{\infty}\frac{1}{|a_{n}|^{\lambda}}<\infty.
$$
More detailed information can be found in \cite{firstab_lit:Eb. Levin}. Denote by $\tilde{\mathfrak{S}}_{\rho}(\mathfrak{H})$ the class of the operators such that
$$
 \tilde{\mathfrak{S}}_{\rho}(\mathfrak{H}):=\{T\in\mathfrak{S}_{\rho+\varepsilon},\,T\, \overline{\in} \,\mathfrak{S}_{\rho-\varepsilon},\,\forall\varepsilon>0 \},
$$
we will call it     {\it Schatten-von Neumann    class  of the convergence exponent.}
We use the following notations
$$
M_{f}(r):=\max\limits_{|z|=r}|f(z)|,\;m_{f}(r):=\min\limits_{|z|=r}|f(z)|,\;z\in \mathbb{C}.
$$
Everywhere further,   unless  otherwise  stated,  we   use  notations of the papers   \cite{firstab_lit:1Gohberg1965},  \cite{firstab_lit:kato1980},  \cite{firstab_lit:kipriyanov1960}, \cite{firstab_lit:1kipriyanov1960},
\cite{firstab_lit:samko1987}.

\vspace{0.5cm}

\noindent{\bf Convergence in the  Abel-Lidsky sense}\\

 In this subsection, we reformulate  results obtained by Lidskii \cite{firstab_lit:1Lidskii} in a more  convenient  form applicable to the reasonings of this paper.   However,  let us begin our narrative.    In accordance with the Hilbert theorem
  (see \cite{firstab_lit:Riesz1955}, \cite[p.32]{firstab_lit:1Gohberg1965})   the spectrum of an arbitrary  compact operator $B$  consists of the so-called normal eigenvalues, it gives us an opportunity to consider a decomposition to a direct sum of subspaces
\begin{equation}\label{1h}
 \mathfrak{H}=\mathfrak{N}_{q} \oplus  \mathfrak{M}_{q},
 \end{equation}
where both  summands are   invariant subspaces regarding the operator $B,$  the first one is  a finite dimensional root subspace corresponding to the eigenvalue $\mu_{q}$ and the second one is a subspace  wherein the operator  $B-\mu_{q} I$ is invertible.  Let $n_{q}$ is a dimension of $\mathfrak{N}_{q}$ and let $B_{q}$ is the operator induced in $\mathfrak{N}_{q}.$ We can choose a basis (Jordan basis) in $\mathfrak{N}_{q}$ that consists of Jordan chains of eigenvectors and root vectors  of the operator $B_{q}.$  Each chain $e_{q_{\xi}},e_{q_{\xi}+1},...,e_{q_{\xi}+k},\,k\in \mathbb{N}_{0},$ where $e_{q_{\xi}},\,\xi=1,2,...,m $ are the eigenvectors  corresponding   to the  eigenvalue $\mu_{q}$   and other terms are root vectors,   can be transformed by the operator $B$ in  accordance  with  the following formulas
\begin{equation}\label{2h}
Be_{q_{\xi}}=\mu_{q}e_{q_{\xi}},\;Be_{q_{\xi}+1}=\mu_{q}e_{q_{\xi}+1}+e_{q_{\xi}},...,Be_{q_{\xi}+k}=\mu_{q}e_{q_{\xi}+k}+e_{q_{\xi}+k-1}.
\end{equation}
Considering the sequence $\{\mu_{q}\}_{1}^{\infty}$ of the eigenvalues of the operator $B$ and choosing a  Jordan basis in each corresponding  space $\mathfrak{N}_{q},$ we can arrange a system of vectors $\{e_{i}\}_{1}^{\infty}$ which we will call a system of the root vectors or following  Lidskii  a system of the major vectors of the operator $B.$
Assume that  $e_{1},e_{2},...,e_{n_{q}}$ is  the Jordan basis in the subspace $\mathfrak{N}_{q}.$  We can prove easily (see \cite[p.14]{firstab_lit:1Lidskii}) that     there exists a  corresponding biorthogonal basis $g_{1},g_{2},...,g_{n_{q}}$ in the subspace $\mathfrak{M}_{q}^{\perp}.$

Using the reasonings \cite{firstab_lit:1kukushkin2021},  we conclude that $\{ g_{i}\}_{1}^{n_{q}}$ consists of the Jordan chains of the operator $B^{\ast}$ which correspond to the Jordan chains  \eqref{2h} due to the following formula
$$
B^{\ast}g_{q_{\xi}+k}= \bar{\mu}_{q} g_{q_{\xi}+k},\;B^{\ast}g_{q_{\xi}+k-1}= \bar{\mu}_{q} g_{q_{\xi}+k-1}+g_{q_{\xi}+k},...,
B^{\ast}g_{q_{\xi}}= \bar{\mu}_{q} g_{q_{\xi}}+g_{q_{\xi}+1}.
$$
   It is not hard to prove   that  the set  $\{g_{\nu}\}^{n_{j}}_{1},\,j\neq i$  is orthogonal to the set $ \{e_{\nu}\}_{1}^{n_{i}}$ (see \cite{firstab_lit:1kukushkin2021}).  Gathering the sets $\{g_{\nu}\}^{n_{j}}_{1},\,j=1,2,...,$ we can obviously create a biorthogonal system $\{g_{n}\}_{1}^{\infty}$ with respect to the system of the major vectors of the operator $B.$ It is rather reasonable to call it as  a system of the major vectors of the operator $B^{\ast}.$  Note that if an element $f\in\mathfrak{H}$ allows a decomposition in the strong sense
$$
f=\sum\limits_{n=1}^{\infty}e_{n}c_{n},\,c_{n}\in \mathbb{C},
$$
then by virtue of  the biorthogonal  system existing, we can claim that such a representation is unique. Further, let us come to the previously made  agrement that the vectors in each Jordan chain are arranged in the same order as in \eqref{2h} i.e.  at the first place there stands an eigenvector. It is clear that under such an assumption  we have
$$
c_{q_{\xi}+i}=\frac{(f,g_{q_{\xi}+k-i})}{(e_{q_{\xi}+i},g_{q_{\xi}+k-i})},\,0\leq i\leq k(q_{\xi}),
$$
where $k(q_{\xi})+1$ is a number of elements in the $q_{\xi}$-th Jourdan chain. In particular, if the vector $e_{q_{\xi}}$ is included to the major system solo, there does not exist a root vector corresponding to the same eigenvalue, then
$$
c_{q_{\xi}}=\frac{(f,g_{q_{\xi}})}{(e_{q_{\xi}},g_{q_{\xi}})}.
$$
Note that in accordance with the property of the biorthogonal sequences, we can expect that the denominators equal to one in the previous two relations.
Consider a formal series corresponding to a decomposition on the major vectors of the operator $B$
$$
f\sim\sum\limits_{n=1}^{\infty}e_{n}c_{n},
$$
where each number $n$ corresponds to a number $q_{\xi}+i$  (thus, the coefficients $c_{n}$ are defined in accordance with the above and  numerated in a simplest way). Consider a formal set of functions   with respect to a real parameter $t$
$$
H_{m}( \varphi,z,t ):=  \frac{e^{ \varphi(z)  t}}{m!} \cdot\frac{d^{m}}{dz^{\,m}}\left\{ e^{-\varphi (z)t}\right\} ,\;m=0,1,2,...\,,\,.
$$
Here we should note that if $\varphi:=z,\,z\in \mathbb{C},$ then we have a set of polynomials, what is in the origin of the concept, see  \cite{firstab_lit:1Lidskii}.  Consider a series
\begin{equation}\label{3h}
\sum\limits_{n=1}^{\infty}c_{n}(t)e_{n},
\end{equation}
where the coefficients $c_{n}(t)$ are defined  in accordance with the correspondence between the indexes  $n$ and $q_{\xi}+i$ in the following way
\begin{equation}\label{4h}
c_{q_{\xi}+i}(t)=   e^{ -\varphi(\lambda_{q})  t}\sum\limits_{m=0}^{k(q_{\xi})-i}H_{m}(\varphi, \lambda_{q},t)c_{q_{\xi}+i+m},\,i=0,1,2,...,k(q_{\xi}),
\end{equation}
here $\lambda_{q}=1/\mu_{q}$ is a characteristic number corresponding to $e_{q_{\xi}}.$
It is clear  that in any case, we have a limit
$
c_{n}(t)\rightarrow \tilde{c}_{n},\,t\rightarrow +\,0,
$
where  a value $\tilde{c}_{n}$ can be calculated directly due to the formula  \eqref{4h}. For instance in the case $\varphi=z ,\,z\in \mathbb{C},$ we have $\tilde{c}_{n}= c_{n}.$
Generalizing the definition given in \cite[p.71]{firstab_lit:Hardy}, we will say that series \eqref{3h} converges  to the element $f$ in the sense $(A,\lambda,\varphi),$ if there exists a sequence of the natural numbers $\{N_{j}\}_{1}^{\infty}$ such that
$$
f=\lim\limits_{t\rightarrow+0}\lim\limits_{j\rightarrow\infty}\sum\limits_{n=1}^{N_{j}}c_{n}(t)e_{n}.
$$
Note that   sums of the latter relation forms a subsequence of the partial sums of the series \eqref{3h}.

  We need the following lemmas \cite{firstab_lit:1Lidskii}, in the adopted form  also see \cite{firstab_lit:1kukushkin2021}. In spite of the fact that the scheme of the Lemma \ref{L3b} proof is the same we present it in the expanded form for the reader convenience.
  Further, considering an arbitrary  compact operator $B: \mathfrak{H}\rightarrow \mathfrak{H}$ such that
$
\Theta(B)\subset \mathfrak{L}_{0}(\theta),\, \theta<\pi,
$
we put the following contour   in correspondence to the operator
\begin{equation*}
\vartheta(B):=\left\{\lambda:\;|\lambda|=r>0,\,|\mathrm{arg} \lambda|\leq \theta+\varsigma\right\}\cup\left\{\lambda:\;|\lambda|>r,\; |\mathrm{arg} \lambda|=\theta+\varsigma\right\},
\end{equation*}
where $\varsigma>0$ is an arbitrary small number, the number $r$ is chosen so that the operator  $ (I-\lambda B)^{-1} $ is regular within the corresponding closed circle. Here, we should note that the compactness property of $B$ gives us the fact   $(I-\lambda B)^{-1}\in \mathcal{B}(\mathfrak{H}),\,\lambda \in \mathbb{C}\setminus \mathrm{ int} \vartheta\, (B).$ It can be proved easily if we note that in accordance with the Corollary 3.3 \cite[p.268]{firstab_lit:kato1980}, we have $\mathrm{P}(B)\subset \mathbb{C}\setminus \overline{\Theta(B)}.$

\begin{lem} \label{L1} Assume that $B$ is a compact  operator,  $\Theta(B)\subset \mathfrak{L}_{0}(\theta),\, \theta<\pi,$ then on each ray $\zeta$ containing the point zero and not belonging to the sector $\mathfrak{L}_{0}(\theta)$ as well as the  real axis, we have
$$
\|(I-\lambda B)^{-1}\|\leq \frac{1}{\sin\psi_{0}},\,\lambda\in \zeta,
$$
where $\,\psi_{0} = \min \{|\mathrm{arg}\zeta -\theta|,|\mathrm{arg}\zeta +\theta|\}.$
\end{lem}

\begin{lem}\label{L2}   Assume that a compact operator $B$ satisfies   the  condition   $B\in \tilde{\mathfrak{S}}_{\rho},$       then
for arbitrary numbers  $R,\kappa$   such that $R>0,\,0<\kappa<1,$ there exists a  circle $|\lambda|=\tilde{R},\,(1-\kappa)R<\tilde{R}<R,$ so that the following estimate holds
$$
\|(I-\lambda B )^{-1}\|_{\mathfrak{H}}\leq e^{\gamma(|\lambda|)|\lambda|^{\varrho}}|\lambda|^{m},\,|\lambda|=\tilde{R},\,m=[\varrho],\,\varrho\geq\rho,
$$
where
$$
\gamma(|\lambda|)= \beta ( |\lambda|^{m+1})  +C \beta(|C  \lambda| ^{m+1}),\;\beta(r )= r^{ -\frac{\varrho}{m+1} }\left(\int\limits_{0}^{r}\frac{n_{B^{m+1}}(t)dt}{t }+
r \int\limits_{r}^{\infty}\frac{n_{B^{m+1}}(t)dt}{t^{ 2  }}\right).
$$
\end{lem}
\begin{lem}\label{L3b} Assume that   $B$ is a compact operator, $\varphi$ is an analytical function inside $\vartheta(B),$ then  in the pole $\lambda_{q}$ of the operator  $(I-\lambda B)^{-1},$ the residue of the vector  function $e^{-\varphi (\lambda) t}B(I-\lambda B)^{-1}\!f,\,(f\in \mathfrak{H}),$  equals to
$$
-\sum\limits_{\xi=1}^{m(q)}\sum\limits_{i=0}^{k(q_{\xi})}e_{q_{\xi}+i}c_{q_{\xi}+i}(t),
$$
where $m(q)$ is a geometrical multiplicity of the $q$-th eigenvalue,  $k(q_{\xi})+1$ is a number of elements in the $q_{\xi}$-th Jourdan chain,
$$
c_{q_{\xi}+j} (t):=e^{-\varphi(\lambda_{q}) t}\sum\limits_{m=0}^{k(q_{\xi})-j} c_{q_{\xi}+j+m} H_{m}(\varphi,\lambda_{q},t ).
$$
\end{lem}
\begin{proof} Consider an integral
$$
 \mathfrak{I}=\frac{1}{2\pi i}\oint\limits_{\vartheta_{q}} e^{-\varphi(\lambda)t}B(I-\lambda B)^{-1}fd\lambda ,\,f\in \mathrm{R}(B),
$$
where the interior of the contour $\vartheta_{q}$ does not contain any poles of the operator $(I-\lambda B)^{-1},$  except of $\lambda_{q}.$ Assume that  $\mathfrak{N}_{q}$ is a  root  space corresponding to $\lambda_{q}$ and  consider  a Jordan basis  \{$e_{q_{\xi}+i}\},\,i=0,1,...,k(q_{\xi}),\;\xi=1,2,...,m(q)$ in $\mathfrak{N}_{q}.$ Using decomposition of the Hilbert space in the direct sum \eqref{1h}, we can represent an element
$$
f=f_{1}+f_{2},
$$
where $f_{1}\in \mathfrak{N}_{q},\;f_{2}\in \mathfrak{M}_{q}.$ Note that the operator function $e^{-\varphi(\lambda)t}B(I-\lambda B)^{-1}f_{2}$ is regular in the interior of the contour $\vartheta_{q},$ it follows from the fact that $\lambda_{q}$ ia a normal eigenvalue (see the supplementary information). Hence, we have
$$
 \mathfrak{I}=\frac{1}{2\pi i}\oint\limits_{\vartheta_{q}} e^{-\varphi(\lambda)t}B(I-\lambda B)^{-1}f_{1}d\lambda.
$$
Using the formula
$$
B(I-\lambda B)^{-1}=\frac{1}{\lambda}\left\{(I-\lambda B)^{-1}-I   \right\}=\frac{1}{\lambda^{2}}\left\{\left(\frac{1}{\lambda}I- B\right)^{-1}-\lambda I   \right\},
$$
we obtain
$$
\mathfrak{I}=-\frac{1}{2\pi i}\oint\limits_{\tilde{\vartheta}_{q}} e^{-\varphi(\zeta^{-1})t}B(\zeta I-  B)^{-1}f_{1}d\zeta,\,\zeta=1/\lambda.
$$
Now, let us decompose the element $f_{1}$ on the corresponding Jordan basis, we have
\begin{equation}\label{5h}
f_{1}=\sum\limits_{\xi=1}^{m(q)}\sum\limits_{i=0}^{k(q_{\xi})}e_{q_{\xi}+i}c_{q_{\xi}+i}.
\end{equation}
In accordance with the  relation \eqref{2h}, we get
$$
Be_{q_{\xi}}=\mu_{q}e_{q_{\xi}},\;Be_{q_{\xi}+1}=\mu_{q}e_{q_{\xi}+1}+e_{q_{\xi}},...,Be_{q_{\xi}+k}=\mu_{q}e_{q_{\xi}+k}+e_{q_{\xi}+k-1}.
$$
Using this formula, we can prove the following relation
\begin{equation}\label{6h}
(\zeta I-  B)^{-1}e_{q_{\xi}+i}=\sum\limits_{j=0}^{i}\frac{e_{q_{\xi}+j}}{(\zeta-\mu_{q})^{i-j+1}}.
\end{equation}
Note that the case $i=0$ is trivial. Consider a case, when $i>0,$ we have
$$
\frac{(\zeta I-  B)e_{q_{\xi}+j}}{(\zeta-\mu_{q})^{i-j+1}}=\frac{\zeta e_{q_{\xi}+j}-Be_{q_{\xi}+j}}{(\zeta-\mu_{q})^{i-j+1}} = \frac{ e_{q_{\xi}+j}}{(\zeta-\mu_{q})^{i-j }}-\frac{e_{q_{\xi}+j-1}}{(\zeta-\mu_{q})^{i-j+1}},\,j>0,
$$
$$
\frac{(\zeta I-  B)e_{q_{\xi} }}{(\zeta-\mu_{q})^{i +1}}=  \frac{ e_{q_{\xi} }}{(\zeta-\mu_{q})^{i }}.
$$
Using these formulas, we obtain
$$
\sum\limits_{j=0}^{i}\frac{(\zeta I-  B)e_{q_{\xi}+j}}{(\zeta-\mu_{q})^{i-j+1}}= \frac{ e_{q_{\xi} }}{(\zeta-\mu_{q})^{i }}+  \frac{ e_{q_{\xi}+1}}{(\zeta-\mu_{q})^{i-1 }}-\frac{e_{q_{\xi} }}{(\zeta-\mu_{q})^{i }}+...
$$
$$
+ \frac{ e_{q_{\xi}+i}}{(\zeta-\mu_{q})^{i-i }}-\frac{e_{q_{\xi}+i-1}}{(\zeta-\mu_{q})^{i-i+1}}=
 \frac{ e_{q_{\xi}+i}}{(\zeta-\mu_{q})^{i-i }},
$$
what gives us the desired result. Now, substituting  \eqref{5h},\eqref{6h}, we get
$$
\mathfrak{I}=- \frac{1}{2\pi i}\sum\limits_{\xi=1}^{m(q)}\sum\limits_{i=0}^{k(q_{\xi})}c_{q_{\xi}+i}\sum\limits_{j=0}^{i} e_{q_{\xi}+j} \oint\limits_{\tilde{\vartheta}_{q}}\frac{ e^{-\varphi(\zeta^{-1})t}}{(\zeta-\mu_{q})^{i-j+1}}d\zeta.
$$
Note that the function $\varphi (\zeta^{-1})$ is analytic inside the interior of $\tilde{\vartheta}_{q},$ hence
$$
 \frac{1}{2\pi i}\oint\limits_{\tilde{\vartheta}_{q}}\frac{ e^{-\varphi(\zeta^{-1})t}}{(\zeta-\mu_{q})^{i-j+1}}d\zeta= \frac{1}{(i-j)!}\lim\limits_{\zeta\rightarrow\, \mu_{q}}\frac{d^{i-j}}{d\zeta^{\,i-j}}\left\{ e^{-\varphi(\zeta^{-1})t}\right\}=:e^{-\varphi(\lambda_{q})t}H_{i-j}(\varphi,\lambda_{q},t ).
$$
Changing the indexes, we have
$$
\mathfrak{I}=-  \sum\limits_{\xi=1}^{m(q)}\sum\limits_{i=0}^{k(q_{\xi})}c_{q_{\xi}+i}e^{-\varphi(\lambda_{q})t}\sum\limits_{j=0}^{i} e_{q_{\xi}+j} H_{i-j}(\varphi,\lambda_{q},t )= -\sum\limits_{\xi=1}^{m(q)}\sum\limits_{j=0}^{k(q_{\xi})}e_{q_{\xi}+j}e^{-\varphi(\lambda_{q})t}\sum\limits_{m=0}^{k(q_{\xi})-j} c_{q_{\xi}+j+m} H_{m}(\varphi,\lambda_{q},t )=
$$
$$
=-\sum\limits_{\xi=1}^{m(q)}\sum\limits_{j=0}^{k(q_{\xi})}e_{q_{\xi}+j}   c_{q_{\xi}+j} (t),
$$
where
$$
c_{q_{\xi}+j} (t):=e^{-\varphi(\lambda_{q})t}\sum\limits_{m=0}^{k(q_{\xi})-j} c_{q_{\xi}+j+m} H_{m}(\varphi,\lambda_{q},t ).
$$
The proof is complete.
\end{proof}
Note that using the reasonings of the last lemma, it is not hard to prove that
$$
c_{q_{\xi}+j} (t)\longrightarrow\sum\limits_{m=0}^{k(q_{\xi})-j}   c_{q_{\xi}+j+m} H_{m}(\varphi,\lambda_{q},0 ),\,t\rightarrow +0.
$$

\section{Main results}

 In this section, we have a challenge how to generalize results \cite{firstab_lit:1kukushkin2021} in the way to make an efficient tool for study abstract evolution equations with the operator function at the right-hand side. The operator function is supposed to be defined on the set of unbounded non-selfadjoint operators.
  First of all,  we consider   statements
with the  necessary  refinement caused by  the involved functions, here we should note that a particular case corresponding to a power function $\varphi$   was considered by Lidskii \cite{firstab_lit:1Lidskii}.  Secondly, we find conditions that guarantee convergence of the involved integral construction and formulate lemmas giving us a tool for further study. We consider   an analogue of the spectral family having involved the operators similar to Riesz projectors (see \cite[p.20]{firstab_lit:1Gohberg1965} ) and using a notion of the  convergence in the Abel-Lidskii sense. As  a main result we prove an existence and uniqueness theorem for evolution equation with the operator function  at the right-hand side. Finally, we discus  an approach   that  we can implement  to apply  the abstract theoretical results  to  concrete  evolution equations.
\begin{lem}\label{L4b} Assume that the operator $B$ satisfies conditions of Lemma \ref{L1},  the entire function $\varphi$ of the order less than a half  maps the inside of the contour $\vartheta(B)$ into the sector $\mathfrak{L}_{0}(\varpi),\,\varpi<\pi/2$  for a sufficiently large value $|z|,\,z\in \mathrm{int}\, \vartheta(B).$ Then the following relation holds
\begin{equation}\label{7a}
\lim\limits_{t\rightarrow+0}\frac{1}{2 \pi i}\int\limits_{\vartheta(B)}e^{-\varphi(\lambda) t}B(I-\lambda B)^{-1}fd\lambda=f,\,f\in \mathrm{R}(B).
\end{equation}
\end{lem}
\begin{proof}
Using the formula
$$
B^{2}(I-\lambda B)^{-1}= \frac{1}{\lambda^{2}}\left\{\left(I- \lambda B\right)^{-1}-(I+\lambda B)   \right\},
$$
we obtain
$$
 \frac{1}{2 \pi i}\int\limits_{\vartheta(B)}e^{-\varphi(\lambda) t}B(I-\lambda B)^{-1}fd\lambda=
 \frac{1}{2 \pi i}\int\limits_{\vartheta(B)}e^{-\varphi(\lambda) t} \lambda^{-2}  \left(I- \lambda B\right)^{-1} Wfd\lambda -
$$
$$
- \frac{1}{2 \pi i}\int\limits_{\vartheta(B)}e^{-\varphi(\lambda) t}\lambda^{-2} (I+\lambda B)  Wfd\lambda =I_{1}(t)+I_{2}(t).
$$
Consider $I_{1}(t).$ Since this improper integral is uniformly convergent regarding $t,$ this fact can be established easily if we apply Lemma \ref{L1}, then using the theorem on the connection with the simultaneous limit and the repeated limit, we get
$$
\lim\limits_{t\rightarrow+0}I_{1}(t)= \frac{1}{2 \pi i}\int\limits_{\vartheta(B)}  \lambda^{-2}  \left(I- \lambda B\right)^{-1} Wfd\lambda.
$$
define a contour  $\vartheta_{R}(B):= \mathrm{Fr}\big\{\{\lambda:\, |\lambda|<R\}\setminus \mathrm{int }\,\vartheta(B) \}\big\}$ and let us prove that
\begin{equation}\label{8a}
 \frac{1}{2\pi i}\oint\limits_{\vartheta_{R}(B)}  \lambda^{-2}  \left(I- \lambda B\right)^{-1} Wfd\lambda\rightarrow \frac{1}{2 \pi i}\int\limits_{\vartheta(B)}  \lambda^{-2}  \left(I- \lambda B\right)^{-1} Wfd\lambda,\;R\rightarrow \infty.
\end{equation}
Consider a decomposition of the contour $\vartheta_{R}(B)$ on terms   $\tilde{\vartheta}_{R}(B):=\{\lambda:\,|\lambda|=R,\, \theta +\varsigma\leq\mathrm{arg} \lambda \leq 2\pi- \theta -\varsigma  \},\,\hat{\vartheta}_{R}:= \{\lambda:\,|\lambda|=r,\,|\mathrm{arg} \lambda |\leq\theta +\varsigma\}\cup
\{\lambda:\,r<|\lambda|<R,\, \mathrm{arg} \lambda  =\theta +\varsigma\}\cup
\{\lambda:\,r<|\lambda|<R,\, \mathrm{arg} \lambda  =-\theta -\varsigma\}.$
 It is clear that
$$
\frac{1}{2\pi i}\oint\limits_{\vartheta_{R}(B)}  \lambda^{-2}  \left(I- \lambda B\right)^{-1} Wfd\lambda =
\frac{1}{2\pi i}\int\limits_{\tilde{\vartheta}_{R}(B)}  \lambda^{-2}  \left(I- \lambda B\right)^{-1} Wfd\lambda+
 $$
 $$
 +\frac{1}{2\pi i}\int\limits_{\hat{\vartheta}_{R}}  \lambda^{-2}  \left(I- \lambda B\right)^{-1} Wfd\lambda.
$$
Let us show that the first summand tends to zero when $R\rightarrow\infty,$ we have
$$
 \left\|\,\int\limits_{\tilde{\vartheta}_{R}(B)}  \lambda^{-2}  \left(I- \lambda B\right)^{-1} Wfd\lambda\right\|_{\mathfrak{H}}\leq  R^{-2}\int\limits_{\theta +\varsigma}^{2\pi - \theta -\varsigma}     \left\|\left(I\lambda^{-1}-  B\right)^{-1} Wf\right\|_{\mathfrak{H}}d\,  \mathrm{arg} \lambda.
$$
Applying Corollary 3.3,  Theorem 3.2  \cite[p.268]{firstab_lit:kato1980}, we have
$$
\left\|\left(I\lambda^{-1}-  B\right)^{-1} \right\|_{\mathfrak{H}}\leq R/\sin \varsigma ,\,\lambda \in \tilde{\vartheta}_{R}(B).
$$
Substituting this estimate to the last integral,
 we obtain the desired result. Thus, taking into account the fact
$$
\frac{1}{2\pi i}\int\limits_{\hat{\vartheta}_{R}}  \lambda^{-2}  \left(I- \lambda B\right)^{-1} Wfd\lambda\rightarrow \frac{1}{2\pi i}\int\limits_{ \vartheta(B)  }  \lambda^{-2}  \left(I- \lambda B\right)^{-1} Wfd\lambda ,\,R\rightarrow \infty,
$$
we obtain  \eqref{8a}. Having noticed that the following integral can be calculated as a residue at the point zero, i.e.
$$
\frac{1}{2\pi i}\oint\limits_{\vartheta_{R}(B)}  \lambda^{-2}  \left(I- \lambda B\right)^{-1} Wfd\lambda =\lim\limits_{\lambda\rightarrow0}\frac{d (I-\lambda B)^{-1} }{d\lambda }W f= f,
$$
we get
$$
\frac{1}{2 \pi i}\int\limits_{\vartheta(B)}  \lambda^{-2}  \left(I- \lambda B\right)^{-1} Wfd\lambda=f.
$$
Hence $I_{1}(t)\rightarrow f,\,t\rightarrow+0.$ Let us show that $I_{2}(t)=0.$ For this purpose, let us  consider  a  contour
  $\vartheta_{R}(B)=\tilde{\vartheta}_{ R}\cup \hat{\vartheta}_{R},$  where    $\tilde{\vartheta}_{ R}:=\{\lambda:\,|\lambda|=R,\,|\mathrm{arg} \lambda |\leq\theta +\varsigma\}$ and  $\hat{\vartheta}_{R}$ is previously defined. It is clear that
$$
\frac{1}{2\pi i}\oint\limits_{\vartheta_{R}(B)}  \lambda^{-2} e^{-\varphi(\lambda)t} \left(I+ \lambda B\right)  Wfd\lambda=\frac{1}{2\pi i}\int\limits_{\tilde{\vartheta}_{ R}}  \lambda^{-2} e^{-\varphi(\lambda)t} \left(I+ \lambda B\right)  Wfd\lambda+
 $$
 $$
+\frac{1}{2\pi i}\int\limits_{\hat{\vartheta}_{R}}  \lambda^{-2} e^{-\varphi(\lambda)t} \left(I+ \lambda B\right)  Wfd\lambda.
$$
Considering the second term  having   taken  into account  the  definition of the improper integral,   we conclude    that if we show that  there exists such a sequence $\{R_{n}\}_{1}^{\infty},\,R_{n}\uparrow\infty$  that
\begin{equation}\label{9a}
\frac{1}{2\pi i}\int\limits_{\tilde{\vartheta}_{ R_{n}}}  \lambda^{-2} e^{-\varphi(\lambda)t} \left(I+ \lambda B\right)  Wfd\lambda\rightarrow 0,\,n\rightarrow \infty,
\end{equation}
then we obtain
\begin{equation}\label{10a}
 \frac{1}{2\pi i}\oint\limits_{\vartheta_{R_{n}}(B)}  \lambda^{-2} e^{-\varphi(\lambda)t} \left(I+ \lambda B\right)  Wfd\lambda\rightarrow \frac{1}{2 \pi i}\int\limits_{\vartheta(B)}  \lambda^{-2} e^{-\varphi(\lambda)t} \left(I+ \lambda B\right)  Wfd\lambda,\;R\rightarrow \infty.
\end{equation}
 Using the lemma conditions, we can accomplish the following estimation
 \begin{equation}\label{11a}
|e^{-\varphi(\lambda)t}|= e^{- \mathrm{Re}\,\varphi(\lambda)t}\leq e^{- C|\varphi(\lambda)|t},\, \lambda \in \tilde{\vartheta}_{ R},
\end{equation}
where $R$ is sufficiently large.  Using the condition imposed upon the order of the entire function and  applying  the Wieman theorem  (Theorem 30 \S 18 Chapter I \cite{firstab_lit:Eb. Levin}),  we can claim that there exists such a sequence $\{R_{n}\}_{1}^{\infty},\,R_{n}\uparrow \infty$ that
\begin{equation*}
\forall \varepsilon>0,\,\exists N(\varepsilon):\,e^{- C|\varphi(\lambda )|t}\leq e^{- C m_{\varphi}(R_{n})t}\leq e^{- C t[M_{\varphi}(R_{n})]^{\cos \pi \phi-\varepsilon}},\,\lambda \in \tilde{\vartheta}_{ R_{n}},\,n>N(\varepsilon),
\end{equation*}
where $\phi$ is the order of the entire  function $\varphi.$
Using this estimate, we get
$$
 \left\|\,\int\limits_{\tilde{\vartheta}_{ R_{n}}}  \lambda^{-2} e^{-\varphi(\lambda)t} \left(I+ \lambda B\right)  Wfd\lambda\right\|_{\mathfrak{H}}\leq C e^{- C t[M_{\varphi}(R_{n})]^{\cos \pi \phi-\varepsilon}}\|Wf\|_{\mathfrak{H}} \int\limits_{-\theta-\varsigma}^{\theta+\varsigma}d\xi.
$$
It is clear that if the order $\phi$ less than a half then we obtain \eqref{9a} and as a consequence \eqref{10a}.   Since the operator function under the integral is analytic, then
$$
\oint\limits_{\vartheta_{R_{n}}(B)}  \lambda^{-2} e^{-\varphi(\lambda)t} \left(I+ \lambda B\right)  Wfd\lambda=0,\,n\in \mathbb{N}.
$$
Combining this relation with \eqref{10a}, we  obtain the fact $I_{2}(t)=0.$ The proof is complete.
\end{proof}
\begin{remark}Note that the statement of the lemma is not true if the order equals zero, in this case we cannot apply the Wieman theorem   (more detailed see the proof of the Theorem 30 \S 18 Chapter I \cite{firstab_lit:Eb. Levin}). At the same time the proof can be easily transformed for the case corresponding to a polynomial function. Here, we should note that the reasonings are the same, we  have to impose   conditions upon the polynomial to satisfy the lemma conditions and  establish an estimate analogous to  \eqref{11a}.   Now assume that $\varphi(z)=c_{0}+c_{1}z+...+c_{n}z^{n},\,z\in \mathbb{C},$ by easy calculations we see that   the condition $$\max\limits_{k=0,1,...,n}(|\mathrm{arg} c_{k}|+k\theta)<\pi/2,$$ gives us  $|\mathrm{arg }\varphi(z)|<\pi/2,\,z\in \mathrm{int} \vartheta(B).$ Thus, we have the fulfilment of the estimate  \eqref{11a}. It can be established easily that $m_{\varphi}(|z|)\rightarrow \infty,\,|z|\rightarrow \infty.$ Combining this fact with \eqref{11a} and preserving the scheme of the reasonings presented in Lemma \ref{L4b}, we obtain \eqref{7a}.
\end{remark}

Bellow, we consider an invertible operator $B$ and use a notation $W:=B^{-1}.$ This agrement is justified by the significance of the operator with a compact resolvent, the detailed information on which spectral properties can be found in the papers   cited in the introduction section.
Consider a function $\varphi$ that can be represented by a Laurent series about the point zero.  Denote by
\begin{equation}\label{12a}
\varphi(W):=\sum\limits_{n=-\infty}^{\infty}  c_{n}  W^{n}
\end{equation}
a formal construction called by a function of the operator, where $c_{n}$ are coefficients corresponding to the function $\varphi.$
The lemma  given bellow are devoted to the study of the conditions under which being imposed the series of operators \eqref{12a} converges on some elements of the Hilbert space $\mathfrak{H},$   thus the operator $\varphi(W)$ is defined.

\begin{lem}\label{5a}Assume that $B$ is a compact  operator,  $\Theta(B)\subset \mathfrak{L}_{0}(\theta),\, \theta<\pi/2,$
\begin{equation}\label{13a}
\varphi(z)=\sum\limits_{n=-\infty}^{s}c_{n}z^{n},\,z\in  \mathbb{C},\,s\in \mathbb{N},\; \max\limits_{n=0,1,...,s}(|\mathrm{arg} c_{n}|+n\theta)<\pi/2,
\end{equation}
{\bf then
\begin{equation}\label{14a}
 \frac{1}{2\pi i}\int\limits_{\vartheta(B)}\varphi(\lambda) e^{-\varphi(\lambda)  t} B(I-\lambda B)^{-1}fd\lambda= \varphi(W)u(t);\;\; \lim\limits_{t\rightarrow +0}\varphi(W)u(t)= \varphi(W)f,
\end{equation}
where}
$$
u(t):=\frac{1}{2\pi i}\int\limits_{\vartheta(B)}  e^{-\varphi(\lambda)  t} B(I-\lambda B)^{-1}fd\lambda,\,f\in \mathrm{D}(W^{s}).
$$
\end{lem}
\begin{proof}
Consider a decomposition of the  Laurent series on  two terms
$$
\varphi_{1}(z)=\sum\limits_{n=0}^{s}c_{n}z^{n};\;\varphi_{2}(z)=\sum\limits_{n=1}^{\infty}c_{-n}z^{-n}.
$$
Consider an obvious relation
\begin{equation}\label{15a}
  \lambda^{k} B^{k}(E-\lambda B)^{-1}=(E-\lambda B)^{-1}-(E+\lambda B+...+\lambda^{k-1}B^{k-1}),\;k\in \mathbb{N}.
\end{equation}
It gives us the following representation
\begin{equation}\label{16x}
\frac{1}{2\pi i}\int\limits_{\vartheta(B)} \lambda^{n} e^{-\varphi(\lambda) t}B(I-\lambda B)^{-1}fd\lambda= I_{1n}(t)
+I_{2n}(t),\;n\in \mathbb{Z}^{-}\cup \{0,1,...,s\},
\end{equation}
 where
$$
I_{1n}:= \frac{1}{2\pi i}\int\limits_{\vartheta(B)}e^{-\varphi(\lambda) t} (I-\lambda B)^{-1}W^{n-1}fd\lambda,\;I_{2n}(t):=0,\,n=0,
$$
\begin{equation*}
 I_{2n}(t):=  \left\{ \begin{aligned}
 -\sum\limits_{k=0}^{n-1 }\beta_{k}(t)B^{k-n+1}  f,\;n> 0,\\
 \sum\limits_{k=-1}^{n}\beta_{k}(t)B^{k-n+1}  f  ,\;  n<0\, \\
\end{aligned}
 \right.,  \;\; \beta_{k}(t):=\frac{1}{2\pi i}\int\limits_{\vartheta(B)}e^{-\varphi(\lambda)t} \lambda^{k}d\lambda.
\end{equation*}
Let us show that  $\beta_{k}(t)=0,$  define a contour  $\vartheta_{R}(B):= \mathrm{Fr}\left\{\mathrm{int }\,\vartheta(B) \,\cap \{\lambda:\,r<|\lambda|<R \}\right\}$ and let us prove that
\begin{equation}\label{17a}
I_{Rk}(t):= \frac{1}{2\pi i}\oint\limits_{\vartheta_{R}(B)}e^{-\varphi(\lambda)t} \lambda^{k}d\lambda\rightarrow \beta_{k}(t),\;R\rightarrow \infty.
\end{equation}
Consider a decomposition of the contour $\vartheta_{R}(B)$ on terms   $\tilde{\vartheta}_{ R}:=\{\lambda:\,|\lambda|=R,\,|\mathrm{arg} \lambda |\leq\theta +\varsigma\}$ and
$\hat{\vartheta}_{R}:= \{\lambda:\,|\lambda|=r,\,|\mathrm{arg} \lambda |\leq\theta +\varsigma\}\cup
\{\lambda:\,r<|\lambda|<R,\, \mathrm{arg} \lambda  =\theta +\varsigma\}\cup
\{\lambda:\,r<|\lambda|<R,\, \mathrm{arg} \lambda  =-\theta -\varsigma\}.$
We have
$$
\frac{1}{2\pi i}\oint\limits_{\vartheta_{R}(B)}e^{-\varphi(\lambda)t} \lambda^{k}d\lambda=
 \frac{1}{2\pi i}\int\limits_{\tilde{\vartheta}_{ R}}e^{-\varphi(\lambda)t} \lambda^{k}d\lambda+
 \frac{1}{2\pi i}\int\limits_{\hat{\vartheta}_{R}}e^{-\varphi(\lambda)t} \lambda^{k}d\lambda.
$$
Having noticed that $I_{Rk}(t)=0,$ since the operator function under the integral is analytic inside the contour, we  come to the conclusion that to obtain the desired result, we should show
\begin{equation}\label{18a}
\frac{1}{2\pi i}\int\limits_{\tilde{\vartheta}_{ R}}e^{-\varphi(\lambda)t} \lambda^{k}d\lambda\rightarrow 0,\;R\rightarrow \infty.
\end{equation}
We have
\begin{equation*}
  \left|\,\int\limits_{\tilde{\vartheta}_{ R}}e^{-\varphi(\lambda)t}\lambda^{k}    d \lambda\,\right| \leq \,R^{k}\int\limits_{\tilde{\vartheta}_{ R}}|e^{-\varphi(\lambda)t}| |d \lambda|\leq R^{k+1}\int\limits_{-\theta-\varsigma}^{\theta+\varsigma}  e^{- t \mathrm{Re}\,\varphi(\lambda) } d \,\mathrm{arg} \lambda.
\end{equation*}
Consider a value   $\mathrm{Re}\,\varphi(\lambda),\, \lambda\in  \tilde{\vartheta}_{ R}$  for a sufficiently large value $R.$ Using the property of the principal part of the Laurent series in is not hard to prove that $\forall\varepsilon>0,\,\exists N(\varepsilon):|\varphi_{2}(\lambda)|<\varepsilon,\,R>N(\varepsilon).$ It follows easily from the condition  \eqref{13a} that $\mathrm{Re}\,\varphi_{1}(\lambda)\geq C|\varphi_{1}(\lambda)|,\,\lambda\in  \tilde{\vartheta}_{ R}.$ It is clear that
$|\varphi_{1}(\lambda)|\sim |c_{s}| R^{s},\,R\rightarrow\infty.$ Thus, we have
\begin{equation}\label{19z}
e^{- t \mathrm{Re}\,\varphi(\lambda) }\leq    e^{- Ct |\varphi(\lambda)| }\leq   e^{-C|\lambda|^{s} t},\,\lambda\in  \tilde{\vartheta}_{R}.
\end{equation}
Applying this estimate, we obtain
$$
\int\limits_{-\theta-\varsigma}^{\theta+\varsigma}  e^{- t \mathrm{Re}\,\varphi(\lambda) } d \,\mathrm{arg} \lambda \leq
  \int\limits_{-\theta-\varsigma}^{\theta+\varsigma}  e^{- C t |\varphi(\lambda)| }  d \,\mathrm{arg} \lambda\leq   e^{- C t  R^{s} }\int\limits_{-\theta-\varsigma}^{\theta+\varsigma}    d \,\mathrm{arg} \lambda.
$$
The latter estimate gives us \eqref{18a} from what follows \eqref{17a}. Therefore $\beta_{k}(t)=0$ and we obtain the fact  $I_{2n}(t)=0.$    Combining the fact of the operator $W$  closedness (see \cite[p.165]{firstab_lit:kato1980} ) with   the definition of the integral in the Riemann sense, we get easily
$$
 W^{n}u(t)=\frac{1}{2\pi i}\int\limits_{\vartheta(B)}e^{-\varphi(\lambda)t} B(I-\lambda B)^{-1}W^{n}fd\lambda,\,n=0,1,...,s.
$$
Thus, using the formula \eqref{16x}, we obtain
\begin{equation*}
 \frac{1}{2\pi i}\int\limits_{\vartheta(B)}\varphi_{1}(\lambda) e^{-\varphi(\lambda)  t} B(I-\lambda B)^{-1}fd\lambda= \varphi_{1}(W)u(t).
\end{equation*}
Consider a principal part of the Laurent series. Using the formula \eqref{16x},  we get for values $n\in \mathbb{N}$
$$
\frac{1}{2\pi i}\int\limits_{\vartheta(B)} \lambda^{-n} e^{-\varphi(\lambda)t}B(I-\lambda B)^{-1}fd\lambda =B^{n}u(t).
$$
 Not that  by virtue of a character of the convergence of the series principal part, we have
$$
  \left\|\sum\limits_{n=1}^{\infty}  c_{-n}   e^{-\varphi(\lambda) t}(I-\lambda B)^{-1}B^{n+1}f\,\right\|_{\mathfrak{H}}\leq C\|f\|_{\mathfrak{H}}\sum\limits_{n=1}^{\infty}  \left|c_{-n}\right| \cdot  \left\|B\right\|_{\mathfrak{H}}^{n+1} <\infty,\;\lambda\in \vartheta(B).
$$
Therefore
$$
 \int\limits_{\vartheta_{j}(B)} \varphi_{2}(\lambda) e^{-\varphi(\lambda)t}B(I-\lambda B)^{-1}fd\lambda=\sum\limits_{n=1}^{\infty}c_{-n}\!\!   \int\limits_{\vartheta_{j}(B)}  e^{-\varphi(\lambda) t}(I-\lambda B)^{-1}B^{n+1}fd \lambda,\,j\in \mathbb{N},
$$
where $$
\vartheta_{j}(B):=\left\{\lambda:\;|\lambda|=r>0,\,|\mathrm{arg} \lambda|\leq \theta+\varsigma\right\}\cup\left\{\lambda:\;r<|\lambda|<r_{j},\,r_{j}\uparrow \infty,\; |\mathrm{arg} \lambda|=\theta+\varsigma\right\}.
$$
Analogously to \eqref{19z}, we can easily get
\begin{equation}\label{20a}
 e^{-\mathrm{Re}\varphi(\lambda) t}\leq e^{-C|\varphi(\lambda)| t}\leq   e^{-C|\lambda|^{s} t},\,\lambda\in \vartheta(B).
\end{equation}
Applying this estimate, we obtain
$$
\left\|\sum\limits_{n=1}^{\infty}c_{-n}\!\!   \int\limits_{\vartheta_{j}(B)}  e^{-\varphi(\lambda) t}(I-\lambda B)^{-1}B^{n+1}fd \lambda\right\|_{\mathfrak{H}}\leq C \|f\|_{\mathfrak{H}}\sum\limits_{n=1}^{\infty}    |c_{-n}| \cdot\|B^{n+1}\|\!\int\limits_{\vartheta_{j}(B)}  e^{- C|\lambda|^{s} t} |d \lambda|\leq
$$
$$
 \leq C \|f\|_{\mathfrak{H}}\sum\limits_{n=1}^{\infty}    |c_{-n}| \cdot\|B\|^{n+1}\int\limits_{\vartheta(B)}  e^{-C|\lambda|^{s} t} |d \lambda|<\infty.
$$
Note that the uniform convergence  of the series  at the left-hand side with respect to $j$ follows from the latter estimate.
Reformulating      the well-known theorem of calculus  on the absolutely convergent series in   terms of the norm, we have
\begin{equation}\label{21a}
\frac{1}{2\pi i}\!\!\int\limits_{\vartheta(B)}\!\! \varphi_{2}(\lambda) e^{-\varphi(\lambda)t}B(I-\lambda B)^{-1}fd\lambda=\frac{1}{2\pi i}\sum\limits_{n=1}^{\infty} c_{-n}\!\!\!  \int\limits_{\vartheta(B)}  e^{-\varphi(\lambda) t} (I-\lambda B)^{-1}B^{n+1}fd \lambda=
 \varphi_{2}(W)u(t).
\end{equation}
Thus, we obtain the first relation \eqref{14a}. Let us establish the second relation \eqref{14a}. Using the formula \eqref{15a}, we  obtain
$$
\frac{1}{2\pi i}\int\limits_{\vartheta(B)} \lambda^{n} e^{-\varphi(\lambda) t}B(I-\lambda B)^{-1}fd\lambda= I_{1n}(t)
+I_{2n}(t),\;n\in \mathbb{Z}^{-}\cup \{0,1,...,s\},
$$
where
$$
I_{1n}(t):= \frac{1}{2\pi i}\int\limits_{\vartheta(B)}e^{-\varphi(\lambda) t}\lambda^{-2}(I-\lambda B)^{-1}W^{n+1}fd\lambda,\;I_{2n}(t):=0,\,n=-2,
$$
\begin{equation*}
 I_{2n}(t):=  \left\{ \begin{aligned}
 -\sum\limits_{k=-2}^{n-1 }\beta_{k}(t)B^{k-n+1}  f,\;n> -2,\\
 \sum\limits_{k=-3}^{n}\beta_{k}(t)B^{k-n+1}  f  ,\;  n\leq-3\, \\
\end{aligned}
 \right. .
\end{equation*}
Using the proved above fact $\beta_{k}(t)=0,$ we have  $I_{2n}(t)=0.$ Since in consequence of Lemma \ref{L1}, inequality \eqref{20a} for arbitrary $j\in \mathbb{N},f\in \mathrm{D}(W^{s}),$ we have
$$
 e^{-\varphi(\lambda) t}\lambda^{-2}(I-\lambda B)^{-1}W^{n+1}f\rightarrow \lambda^{-2}(I-\lambda B)^{-1}W^{n+1}f,\,t\rightarrow +0 ,\,\lambda\in \vartheta_{j}(B),
$$
where convergence is uniform   with respect to the variable $\lambda,$    the  improper  integral $I_{1n}(t)$
 is uniformly convergent with respect to the variable $t,$ then
  we get
\begin{equation*}
I_{1n}(t)\rightarrow \frac{1}{2\pi i}\int\limits_{\vartheta(B)} \lambda^{-2}(I-\lambda B)^{-1}W^{n+1}fd\lambda,\,t\rightarrow+0,
\end{equation*}
Note that the last integral can be calculated as a residue, we have
\begin{equation}\label{22a}
 \frac{1}{2\pi i}\int\limits_{\vartheta(B)} \lambda^{-2}(I-\lambda B)^{-1}W^{n+1}fd\lambda=
  \lim\limits_{\lambda\rightarrow0}\frac{d (I-\lambda B)^{-1} }{d\lambda }W^{n+1}f=W^{n}f,
$$
$$
 n\in \mathbb{Z}^{-}\cup \{0,1,...,s\}.
\end{equation}
It is obvious that using this  formula, we obtain the following relation
\begin{equation*}
\lim\limits_{t\rightarrow+0}\frac{1}{2\pi i}\int\limits_{\vartheta(B)}\varphi_{1}(\lambda) e^{-\varphi(\lambda) t} B(I-\lambda B)^{-1}fd\lambda= \sum\limits_{n=0}^{s}c_{n}W^{n}f ,\;f\in \mathrm{D}(W^{s}) .
\end{equation*}
Consider a principal part of the Laurent series.
The following reasonings are analogous to the above,   we get
$$
\left\| \sum\limits_{n=1}^{\infty}   c_{-n} \!\!\int\limits_{\vartheta(B)}   e^{-\varphi(\lambda) t}\lambda^{-2} (I-\lambda B)^{-1}B^{n-1}fd\lambda\right\|\leq C \|f\|_{\mathfrak{H}}\sum\limits_{n=1}^{\infty}    |c_{-n}| \cdot\|B^{n-1}\|\int\limits_{\vartheta(B)} |\lambda|^{-2}  e^{-C|\lambda|^{s} t} |d \lambda|\leq
$$
$$
 \leq C \|f\|_{\mathfrak{H}}\sum\limits_{n=1}^{\infty}    |c_{-n}| \cdot\|B\|^{n-1}\int\limits_{\vartheta(B)}  |\lambda|^{-2} |d \lambda|<\infty.
$$
It gives us  the uniform convergence  of the series with respect to $t$  at the left-hand side of the last relation.
Using the analog of   the well-known theorem of calculus  on the absolutely convergent series, we have
$$
 \sum\limits_{n=1}^{\infty} c_{-n} \!\!  \int\limits_{\vartheta(B)}   e^{-\varphi(\lambda) t}\lambda^{-2} (I-\lambda B)^{-1}B^{n-1}fd\lambda\rightarrow
  \sum\limits_{n=1}^{\infty} c_{-n}\!\!   \int\limits_{\vartheta(B)}    \lambda^{-2} (I-\lambda B)^{-1}B^{n-1}fd\lambda,\;t\rightarrow+0.
$$
  {\bf Taking into account  \eqref{21a}, \eqref{22a},  we get
$$
\lim\limits_{t\rightarrow+0}\frac{1}{2\pi i}\int\limits_{\vartheta(B)}\varphi_{2}(\lambda) e^{-\lambda t} B(I-\lambda B)^{-1}fd\lambda= \sum\limits_{n=1}^{\infty}c_{-n}B^{n}f ,\;f\in \mathfrak{H}.
$$
It is clear that  the second relation \eqref{14a} holds. The proof is complete.}
\end{proof}

\noindent{\bf Existence and uniqueness theorems}\\

This paragraph is a climax of the paper, here we represent a theorem that put a beginning of a marvelous research based on the Abell-Lidskii method. The attempt to consider an operator function at the right-hand side was made in the paper  \cite{firstab_lit:2kukushkin2022}, where we consider a case that is not so difficult  since  the corresponding function is of the power type. In contrast, in this paper we consider a more complicated case, a function that compels us to involve a principally different method of study. The existence and uniqueness theorem given bellow is based on the one of the number of  theorems  presented in \cite{firstab_lit:1kukushkin2021}.

 Further, we will consider a Hilbert space $\mathfrak{H}$ consists of   element-functions $u:\mathbb{R}_{+}\rightarrow \mathfrak{H},\,u:=u(t),\,t\geq0$    and we will assume that if $u$ belongs to $\mathfrak{H}$    then the fact  holds for all values of the variable $t.$ Notice that under such an assumption all standard topological  properties as completeness, compactness e.t.c. remain correctly defined. We understand such operations as differentiation and integration in the generalized sense that is caused by the topology of the Hilbert space $\mathfrak{H}.$ The derivative is understood as the following  limit
$$
  \frac{u(t+\Delta t)-u(t)}{\Delta t}\stackrel{\mathfrak{H}}{ \longrightarrow}\frac{du}{dt} ,\,\Delta t\rightarrow 0.
$$
Let $t\in \Omega:=[a,b],\,0< a <b<\infty.$ The following integral is understood in the Riemann  sense as a limit of partial sums
\begin{equation*}
\sum\limits_{i=0}^{n}u(\xi_{i})\Delta t_{i}  \stackrel{\mathfrak{H}}{ \longrightarrow}  \int\limits_{\Omega}u(t)dt,\,\lambda\rightarrow 0,
\end{equation*}
where $(a=t_{0}<t_{1}<...<t_{n}=b)$ is an arbitrary splitting of the segment $\Omega,\;\lambda:=\max\limits_{i}(t_{i+1}-t_{i}),\;\xi_{i}$ is an arbitrary point belonging to $[t_{i},t_{i+1}].$
The sufficient condition of the last integral existence is a continuous property (see\cite[p.248]{firstab_lit:Krasnoselskii M.A.}) i.e.
$
u(t)\stackrel{\mathfrak{H}}{ \longrightarrow}u(t_{0}),\,t\rightarrow t_{0},\;\forall t_{0}\in \Omega.
$
The improper integral is understood as a limit
\begin{equation*}
 \int\limits_{a}^{b}u(t)dt\stackrel{\mathfrak{H}}{ \longrightarrow} \int\limits_{a}^{c}u(t)dt,\,b\rightarrow c,\,c\in [0,\infty].
\end{equation*}
Let us study   a Cauchy problem
\begin{equation}\label{23a}
     \frac{du}{d t} +\varphi(W) u=0  ,\;u(0)=h\in \mathrm{D}(W),
\end{equation}
in the case  when the operator   $\varphi(W)$ is   accretive  we assume  that   $h\in \mathfrak{H}.$

\begin{teo}\label{T1} Assume that $B$ is a compact operator, $\Theta(B) \subset   \mathfrak{L}_{0}(\theta),\,\theta< \pi/2   ,\;B\in\tilde{\mathfrak{S}}_{\rho},$   moreover  in the case $B  \in \tilde{\mathfrak{S}}_{\rho}\setminus  \mathfrak{S}_{\rho}$ the additional condition holds
\begin{equation}\label{24a}
   \frac{n_{B^{m+1}}(r^{m+1})}{r^{\rho} }\rightarrow 0,\, m=[\rho],
\end{equation}
the function $\varphi$ is satisfied the conditions of Lemma \ref{5a}, the following  condition holds $s>\rho.$
Then a sequence of natural numbers $\{N_{\nu}\}_{0}^{\infty}$ can be chosen so that
there exists a solution of the Cauchy problem \eqref{23a} in the form
\begin{equation}\label{25a}
u(t)= \frac{1}{2\pi i}\int\limits_{\vartheta(B)}e^{-\varphi(\lambda) t}B(I-\lambda B)^{-1}h\, d \lambda
=  \sum\limits_{\nu=0}^{\infty}  \sum\limits_{q=N_{\nu}+1}^{N_{\nu+1}}\sum\limits_{\xi=1}^{m(q)}\sum\limits_{i=0}^{k(q_{\xi})}e_{q_{\xi}+i}c_{q_{\xi}+i}(t),
\end{equation}
where
\begin{equation}\label{26}
  \sum\limits_{\nu=0}^{\infty}\left\|\sum\limits_{q=N_{\nu}+1}^{N_{\nu+1}}\sum\limits_{\xi=1}^{m(q)}
  \sum\limits_{i=0}^{k(q_{\xi})}e_{q_{\xi}+i}c_{q_{\xi}+i}(t)\right\|_{\mathfrak{H}}<\infty.
\end{equation}
 Moreover, the existing solution is unique if the operator   $\varphi(W)$ is accretive.
 \end{teo}

\begin{proof}
Firstly, let us establish relation \eqref{26}.  Consider a contour $\vartheta(B).$ Having fixed $R>0,0<\kappa<1,$ so that $R(1-\kappa)=r,$ consider a monotonically increasing sequence $\{R_{\nu}\}_{0}^{\infty},\,R_{\nu}=R(1-\kappa)^{-\nu+1}.$   Using Lemma \ref{L2}, we get
$$
\|(I-\lambda B )^{-1}\|_{\mathfrak{H}}\leq e^{\gamma (|\lambda|)|\lambda|^{\rho}}|\lambda|^{m},\,m=[\rho],\,|\lambda|=\tilde{R}_{\nu},\;R_{\nu}<\tilde{R}_{\nu}<R_{\nu+1},
$$
where the function  $\gamma (r)$ is defined in Lemma \ref{L2},
$$
\beta(r )= r^{ -\frac{\rho}{m+1} }\left(\int\limits_{0}^{r}\frac{n_{B^{m+1}}(t)}{t }dt+
r \int\limits_{r}^{\infty}\frac{n_{B^{m+1}}(t)}{t^{ 2  }}dt\right).
$$
Note that  in accordance with Lemma 3 \cite{firstab_lit:1Lidskii}  the following relation holds
\begin{equation}\label{27}
\sum\limits_{i=1}^{\infty}\lambda^{\frac{\rho+\varepsilon}{ (m+1)}}_{i}( \tilde{B} )\leq \sum\limits_{i=1}^{\infty}s^{ \,\rho+\varepsilon }_{i}( B )<\infty,\,\varepsilon>0,
\end{equation}
where   $\tilde{B}:=(B^{\ast m+1}A^{m+1})^{1/2}.$  It is clear that   $\tilde{B}\in \tilde{\mathfrak{S}}_{\upsilon},\,\upsilon\leq \rho/(m+1).$
Denote by $\vartheta_{\nu}$ a bound of the intersection of the ring $\tilde{R}_{\nu}<|\lambda|<\tilde{R}_{\nu+1}$ with the interior of the contour $\vartheta(B),$ denote by $N_{\nu}$ a number of poles being   contained  in the set $\mathrm{int }\,\vartheta(B) \,\cap \{\lambda:\,r<|\lambda|<\tilde{R}_{\nu} \}.$ In accordance with Lemma \ref{L3b},  we get
\begin{equation}\label{28}
 \frac{1}{2\pi i}\oint\limits_{\vartheta_{\nu}}e^{-\varphi(\lambda) t} B(I-\lambda B)^{-1}h d \lambda =\sum\limits_{q=N_{\nu}+1}^{N_{\nu+1}}\sum\limits_{\xi=1}^{m(q)}\sum\limits_{i=0}^{k(q_{\xi})}e_{q_{\xi}+i}c_{q_{\xi}+i}(t),\;h\in \mathfrak{H}.
\end{equation}
Let us estimate the above integral, for this purpose split the contour $\vartheta_{\nu}$ on  terms $\tilde{\vartheta}_{ \nu  }:=\{\lambda:\,|\lambda|=\tilde{R}_{\nu},\,|\mathrm{arg} \lambda |\leq\theta +\varsigma\},\,\tilde{\vartheta}_{ \nu+1  },\, \vartheta_{\nu_{+}}:=
\{\lambda:\,\tilde{R}_{\nu}<|\lambda|<\tilde{R}_{\nu+1},\, \mathrm{arg} \lambda  =\theta +\varsigma\},\,\vartheta_{\nu_{-}}:=
\{\lambda:\,\tilde{R}_{\nu}<|\lambda|<\tilde{R}_{\nu+1},\, \mathrm{arg} \lambda  =-\theta -\varsigma\}.$   Applying  Lemma \ref{L2}, relation \eqref{19z},  we get
\begin{equation*}
 J_{  \nu  }: =\left\|\,\int\limits_{\tilde{\vartheta}_{ \nu }}e^{-\varphi(\lambda) t} B(I-\lambda B)^{-1}h d \lambda\,\right\|_{\mathfrak{H}}\leq \,
 \int\limits_{\tilde{\vartheta}_{ \nu }}e^{-t\mathrm{Re}\,\varphi(\lambda)  } \left\|B(I-\lambda B)^{-1}h \right\|_{\mathfrak{H}} |d \lambda|\leq
\end{equation*}
$$
 \leq e^{\gamma (|\lambda|)|\lambda|^{\rho} }|\lambda|^{m+1}   Ce^{-  C|\lambda|^{s}t } \int\limits_{-\theta-\varsigma}^{\theta+\varsigma}  d \,\mathrm{arg} \lambda,\,|\lambda|=\tilde{R}_{\nu}.
$$
Thus,  we get
$
J_{ \nu } \leq     Ce^{\gamma (|\lambda|)|\lambda|^{\rho}-C|\lambda|^{s}t   }|\lambda|^{m+1},
   $
   where
   $
   \,m=[\rho],\,|\lambda|=\tilde{R}_{\nu}.
$
Let us show that for a fixed $t$ and  a sufficiently large $|\lambda|,$ we have  $ \gamma (|\lambda|)|\lambda|^{\rho}-C|\lambda|^{s}t     < 0.$
It follows  directly from Lemma 2 \cite{firstab_lit:1kukushkin2021}, we should consider  \eqref{27}, in the case when  $B\in \mathfrak{S}_{\rho}$ as well as in the case $B  \in \tilde{\mathfrak{S}}_{\rho}\setminus  \mathfrak{S}_{\rho}$  but here we must involve the additional condition \eqref{24a}.
 Therefore
$$
 \sum\limits_{\nu=0}^{\infty}J_{\nu}<\infty.
$$
Using the analogous estimates, applying  Lemma \ref{L1}, we get
$$
 J^{+}_{\nu}: =\left\|\,\int\limits_{\vartheta_{\nu_{+}}}e^{-\varphi(\lambda) t} B(I-\lambda B)^{-1}h d \lambda\,\right\|_{\mathfrak{H}}\leq  C\|h\|_{\mathfrak{H}} \cdot C\int\limits_{R_{\nu}}^{R_{\nu+1}}  e^{-  C t \mathrm{Re}\, \varphi(\lambda) }   |d   \lambda|\leq
  $$
  $$
  \leq C  e^{-t CR^{m}_{\nu}  }     \int\limits_{R_{\nu}}^{R_{\nu+1}}   |d   \lambda|=
 C  e^{-t CR^{m}_{\nu}  }  \{R_{\nu+1}-R_{\nu} \}.
 $$
$$
 J^{-}_{\nu}: =\left\|\,\int\limits_{\vartheta_{\nu_{-}}}e^{-\varphi(\lambda) t}B(I-\lambda B)^{-1}h d \lambda\,\right\|_{\mathfrak{H}}\leq   C  e^{-t CR^{m}_{\nu}  }     \int\limits_{R_{\nu}}^{R_{\nu+1}}   |d   \lambda|=
 C  e^{-t CR^{m}_{\nu}  }  \{R_{\nu+1}-R_{\nu} \}.
$$
The obtained results allow us to claim (the proof is left to the reader) that
$$
  \sum\limits_{\nu=0}^{\infty}J^{+}_{\nu}<\infty,\;\; \sum\limits_{\nu=0}^{\infty}J^{-}_{\nu}<\infty.
$$
Using the formula \eqref{28}, the given above decomposition of the contour $\vartheta_{\nu},$  we   obtain the relation \eqref{26}. Let us establish   \eqref{25a}, for this purpose, we should note that in accordance with  relation \eqref{28}, the properties of the contour integral, we have
 $$
 \frac{1}{2\pi i}\!\!\oint\limits_{\vartheta_{\tilde{R}_{p}}  (B)}\!\!\!\!e^{-\varphi(\lambda) t} B(I-\lambda B)^{-1}h \,d \lambda =
   \sum\limits_{\nu=0}^{p-1} \sum\limits_{q=N_{\nu}+1}^{N_{\nu+1}}\sum\limits_{\xi=1}^{m(q)}
  \sum\limits_{i=0}^{k(q_{\xi})}e_{q_{\xi}+i}c_{q_{\xi}+i}(t)
  ,\;h\in \mathfrak{H},\,p\in \mathbb{N},
$$
where the contour   $\vartheta_{\tilde{R}_{p}}  (B)$ is defined in Lemma \ref{5a}.    Using the proved above  fact $J_{ \nu }\rightarrow0,\,\nu\rightarrow\infty,$    we can easily  get
$$
\frac{1}{2\pi i}\!\!\oint\limits_{\vartheta_{\tilde{R}_{p}}  (B)}\!\!\!\!e^{-\varphi(\lambda) t} B(I-\lambda B)^{-1}h \,d \lambda\rightarrow \frac{1}{2\pi i}\oint\limits_{\vartheta  (B)} e^{-\varphi(\lambda) t} B(I-\lambda B)^{-1}h \,d \lambda,\;p\rightarrow\infty.
$$
 The latter relation
  gives us the desired result \eqref{25a}.
Let us show that $u(t)$ is a solution of the problem \eqref{23a}. Applying Lemma \ref{5a}, we get
$$
  \varphi(W)u(t)=\frac{1}{2\pi i}\int\limits_{\vartheta(B)}\varphi(\lambda) e^{-\varphi(\lambda)  t} B(I-\lambda B)^{-1}hd\lambda.
$$
Now,  we need establish  the following relation
\begin{equation}\label{29}
\frac{d u}{d t}=-\frac{1}{2\pi i} \int\limits_{\vartheta(B)}\varphi(\lambda) e^{-\varphi(\lambda)  t} B(I-\lambda B)^{-1}h \,d\lambda,\,h\in \mathfrak{H},
\end{equation}
i.e. we can use a differentiation operation  under the integral.
For  this purpose, let us   prove that for an arbitrary $\vartheta_{j}(B)$ (the definition is given in Lemma \ref{5a}) there exists a limit
\begin{equation}\label{30}
  \frac{e^{-\varphi(\lambda)  \Delta t} -1}{ \Delta t}\, e^{-\varphi(\lambda)  t}B(I-\lambda B)^{-1}h\stackrel{\mathfrak{H}}{ \longrightarrow}-\varphi(\lambda) e^{-\varphi(\lambda)  t}B(I-\lambda B)^{-1}h ,\,\Delta t\rightarrow 0,
\end{equation}
where convergence is uniform with respect to $ \lambda\in \vartheta_{j}(B).$ Applying Lemma \ref{L1},   we get
$$
\left\| \frac{e^{-\varphi(\lambda)  \Delta t} -1}{ \Delta t} e^{-\varphi(\lambda)  t}B(I-\lambda B)^{-1}h +\varphi(\lambda) e^{-\varphi(\lambda)  t}B(I-\lambda B)^{-1}h\right\|_{\mathfrak{H}}\leq
$$
$$
\leq C  \left|\frac{e^{-\varphi(\lambda)  \Delta t} -1}{ \Delta t} +\varphi(\lambda) \right|\max\limits_{\lambda\in \vartheta_{j}(B)}\!\!e^{-\mathrm{Re}\,\varphi(\lambda)  t}.
$$
It is clear that
$$
\frac{e^{-\varphi(\lambda)  \Delta t} -1}{ \Delta t} \rightarrow-\varphi(\lambda),\,\Delta t\rightarrow 0,
$$
where convergence, in accordance with the Heine-Cantor theorem, is uniform with respect to $\lambda\in \vartheta_{j}(B).$    Thus, we obtain \eqref{30}.   Using decomposition on the Taylor series, applying \eqref{20a}, we get
$$
\left\|\frac{e^{-\varphi(\lambda)  \Delta t} -1}{ \Delta t}e^{-\varphi(\lambda)  t}\right\|_{\mathfrak{H}}\leq |\varphi(\lambda)| e^{ |\varphi(\lambda)  \Delta t|}
e^{-\mathrm{Re}\,\varphi(\lambda)  t}\leq |\varphi(\lambda)| e^{(\Delta t-Ct)  |\varphi(\lambda)|  }\leq
 $$
 $$
 \leq|\varphi(\lambda)| e^{(\Delta t-Ct) C | \lambda |^{s}  },\,\lambda\in \vartheta(B).
$$
Thus applying the latter estimate,  Lemma \ref{L1}, for a sufficiently small value $\Delta t,$ we get
$$
 \left\|\, \int\limits_{\vartheta(B)}\frac{e^{-\varphi(\lambda)  \Delta t} -1}{ \Delta t} e^{-\varphi(\lambda)  t}B(I-\lambda B)^{-1}h d\lambda \right\|_{\mathfrak{H}}
  \leq
   C \|h\|_{\mathfrak{H}} \int\limits_{\vartheta(B)} e^{-C   |\lambda|^{s}}|\lambda|^{s}   |d\lambda|.
$$
The function under the integral at the right-hand side of the last relation guaranties that the improper integral at the left-hand side
 is uniformly convergent with respect to $\Delta t.$ These facts give us an opportunity to claim that the relation \eqref{29} holds. Here, we should explain that this conclusion is based upon the generalization of the well-known theorem of the calculus. In its own turn it follows easily from the theorem on the connection with the simultaneous limit and the repeated limit. We left a complete  investigation of the matter to the reader having noted that the scheme of the reasonings is absolutely the same in comparison with the ordinary calculus.
    Thus, we obtain the fact that $u$ is a solution of the  equation \eqref{23a}.

 Let us show that the initial condition holds in the sense
$
u(t)   \xrightarrow[   ]{\mathfrak{H}}  h,\,t\rightarrow+0.
$
It becomes clear in the case  $h\in \mathrm{D}( W ),$ for in this case  it suffices to apply Lemma \ref{L4b}, what  gives  us the desired result, i.e. we can    put $u(0)=h.$  Consider a case when $h$ is an arbitrary element of the Hilbert space $\mathfrak{H}$  and let us involve the accretive property of the operator $ \varphi(W).$ In accordance with the above, for a fixed value of $t,$ we can understand a correspondence between $u(t)$ and $h$ as an operator  $S_{t}:\mathfrak{H}\rightarrow \mathfrak{H}.$
 Let us prove  that
$
\|S_{t}\|_{\mathfrak{H}\rightarrow \mathfrak{H}}\leq1,\;t>0.
$
Firstly, assume   that $h\in \mathrm{D}( W ).$
Let us multiply the both sides of the  relation \eqref{23a} on $u$  in the sense of the inner product, we get
$
\left(u'_{t},u\right)_{\mathfrak{H}}+\left( \varphi(W) u,u\right)_{\mathfrak{H}}=0.
$
Consider a real part of the last relation, we have
$
\mathrm{Re}\left(u'_{t},u\right)_{\mathfrak{H}}+\mathrm{Re}(\varphi(W)u,u)_{\mathfrak{H}}= \left(u'_{t},u\right)_{\mathfrak{H}}/2+ \left(u, u'_{t}\right)_{\mathfrak{H}}/2+\mathrm{Re}(\varphi(W)u,u)_{\mathfrak{H}}.
$
Therefore
$
   \left(\|u(t)\|_{\mathfrak{H}}^{2}\right)'_{t}  =-2\mathrm{Re}(\varphi(W)u,u)_{\mathfrak{H}}\leq 0.
$
Integrating both sides from zero to $\tau>0,$ we get
$
  \|u(\tau)\|_{\mathfrak{H}}^{2}-  \|u(0)\|_{\mathfrak{H}}^{2} \leq 0.
$
The last relation can be rewritten in the form
$
\|S_{t}h\|_{\mathfrak{H}}\leq \|h\|_{\mathfrak{H}},\,h\in \mathrm{D}( W ).
$
Since $\mathrm{D}( W )$ is a dense set in $ \mathfrak{H},$   we obviously  obtain  the  desired result, i.e. $\|S_{t}\|_{\mathfrak{H}\rightarrow \mathfrak{H}}\leq 1.$
 Now, having assumed that
$
h_{n}   \xrightarrow[   ]{\mathfrak{H}}  h,\,n\rightarrow \infty,\;\{h_{n}\}\subset \mathrm{D}( W ),\,h\in \mathfrak{H},
$
consider the following reasonings
$
\|u(t)-h\|_{\mathfrak{H}}=\|S_{t}h-h\|_{\mathfrak{H}}=\|S_{t}h-S_{t}h_{n}+S_{t}h_{n}-h_{n}+h_{n}-    h\|_{\mathfrak{H}}\leq \|S_{t}\|\cdot\|h- h_{n}\|_{\mathfrak{H}}+\|S_{t}h_{n}-h_{n}\|_{\mathfrak{H}}+\|h_{n}-    h\|_{\mathfrak{H}}.
$
Note that
$
S_{t}h_{n}   \xrightarrow[   ]{\mathfrak{H}}  h_{n},\,t\rightarrow+0.
$
It is clear that  if we choose $n$ so that  $\|h- h_{n}\|_{\mathfrak{H}}<\varepsilon/3$ and  after that  choose $t$   so that $\|S_{t}h_{n}-h_{n}\|_{\mathfrak{H}}<\varepsilon/3,$ then we obtain
 $\forall\varepsilon>0,\,\exists \delta(\varepsilon):\,\|u(t)-h\|_{\mathfrak{H}}<\varepsilon,\,t<\delta.$
   Thus, we can put  $u(0)=h$ and claim that   the initial condition holds in the case $h\in \mathfrak{H}.$       The uniqueness follows easily from the fact that $ \varphi(W) $ is accretive.  In this case, repeating the previous reasonings, we come to
$
  \|\phi(\tau)\|_{\mathfrak{H}}^{2} \leq  \|\phi(0)\|_{\mathfrak{H}}^{2},
$
where $\phi $ is a sum of    solutions $u_{1} $ and $u_{2}.$ Notice that  by virtue of the initial conditions, we have  $\phi(0)=0.$  Therefore,   the previous    relation   can hold only if $\phi =0.$  The proof is complete.
\end{proof}
\begin{remark}Note that generally the existence and uniqueness  Theorem \ref{T1} is based upon  the Theorem 2 \cite{firstab_lit:1kukushkin2021}. The corresponding analogs based upon the Theorems 3,4 \cite{firstab_lit:1kukushkin2021} can be obtained due to the same scheme and the proofs are not worth representing. At the same time the mentioned analogs  can be useful because of   special conditions imposed upon the operator $B$ such as ones formulated in terms of the operator order \cite{firstab_lit:1kukushkin2021}. Here we should also appeal to an artificially constructed normal operator presented in \cite{firstab_lit:2kukushkin2022}.
\end{remark}

\noindent{\bf Concrete operators}\\

It is remarkable that  the made approach allows us to obtain a solution  analytically  for the right-hand  side  -- a function of an operator belonging to a sufficiently wide class of operators. A plenty of examples  are presented in the paper  \cite{firstab_lit:2kukushkin2022} where such well-known operators as  the Riemann-Liouville fractional differential operator, the Kipriyanov operator, the Riesz potential, the difference operator are considered.
 An interesting example can be also  found in the paper  \cite{firstab_lit(arXiv non-self)kukushkin2018}.
 More general approach, implemented in the paper \cite{kukushkin2021a} allows us to built a transform of an operator belonging to the class of  m-accretive operators.   We should stress a significance of the last claim  for the  class    contains the   infinitesimal generator of a $C_{0}$   semigroup of contractions. In its own turn  fractional differential operators of the real order can be expressed in terms of the  infinitesimal generator of the corresponding semigroup,  what makes the offered generalization relevant (more detailed see \cite{kukushkin2021a}). Bellow, we present a rather abstract example for which the paper results can be applied.
 Consider a transform of an m-accretive operator $J$ acting in $\mathfrak{H}$
$$
 Z^{\alpha}_{G,F}(J):= J^{\ast}GJ+FJ^{\alpha},\,\alpha\in [0,1),
$$
where symbols  $G,F$  denote  operators acting in $\mathfrak{H}.$
 The   Theorem 5 \cite{kukushkin2021a} gives us a tool to describe spectral properties  of the  transform $Z^{\alpha}_{G,F}(J).$ Particularly, we can establish the order  of the transform and its belonging to the  Schatten-von Neumann class of the convergence  exponent by virtue of the Theorem 3 \cite{kukushkin2021a}. Thus, having known the index of the Schatten-von Neumann class of the convergence  exponent, we can apply Theorem \ref{T1} to the transform.

\section{Conclusions}

In this paper,  we invented a technique to study evolution equations with the right-hand side   a function of the non-selfadjoint unbounded operator.   It is remarkable that,  we may say that the main issue of the paper is an application of  the spectral theorem to the special class of non-selfadjoint  operators and in the natural way,  we come to the definition of a function of the unbounded non-selfadjoint operator. Under this point of view,  the main  highlights of this paper  are  propositions analogous to the spectral theorem.  We can perceive them  as an introduction or a  way of reformulating the main principles of the spectral theorem based upon  the peculiarities of the convergence in the Abell-Lidsky sense. In this regard,   the main obstacle that appears   is how to define  an analogue of a  spectral family or decomposition of the identical operator.

However, as  a main result we have obtained an approach  allowing us to principally broaden conditions imposed upon the right-hand side of the evolution equation in the abstract Hilbert space. The application part of the paper appeals to the theory of differential equations.  In particular,     the existence and uniqueness theorems  for evolution  equations,   with the right-hand side being presented by  an operator function of  a  differential operator  with a fractional derivative in  final terms, are covered by the invented abstract method. In connection with this,  such operators  as a Riemann-Liouville  fractional differential operator,    Kipriyanov operator, Riesz potential,  difference operator can be considered. Moreover, we can consider a class the artificially constructed   normal operators   for which the clarification of the Lidskii  results  relevantly works.
Apparently, the further step in the theoretical study  may  be  to consider an entire function that generates the operator function, note that a prerequisite of the prospective result  is given by Lemma \ref{L4b}. In this case, some difficulties that may appear relate to     the propositions analogous to Lemma \ref{5a} and Theorem \ref{T1}.   Having been inspired by the above ideas,  we  hope  that  the concept  will have a further development.

\end{document}